\def\marker{\>\hbox{${\vcenter{\vbox{
    \hrule height 0.4pt\hbox{\vrule width 0.4pt height 6pt
    \kern6pt\vrule width 0.4pt}\hrule height 0.4pt}}}$}\>}

\documentclass[12pt]{article}

\usepackage[left=2cm,right=2cm,top=3cm,bottom=3cm]{geometry}

\usepackage{amsmath, amssymb, latexsym, amsthm}
\usepackage{hyperref}
\usepackage[hang,flushmargin]{footmisc}

\usepackage{lpic}
\newtheorem{theorem}{Theorem} 
\newtheorem{theorem*}{Theorem} 
\newtheorem{proposition}{Proposition} 
\newtheorem{corollary}[theorem]{Corollary}
\newtheorem{lemma}{Lemma}
\newtheorem{conjecture}[theorem]{Conjecture}
\newtheorem{fact}[theorem]{Fact}
\newtheorem*{hypothesis}{Hypothesis}

\theoremstyle{definition}

\newtheorem{observation}{Observation}
\newtheorem{problem}[theorem]{Problem}

\theoremstyle{remark}

{\end{enumerate}}

\newcounter{casenum}	
\newcounter{subcasenum}	
\numberwithin{subcasenum}{casenum}
\newcounter{subsubcasenum}	
\numberwithin{subsubcasenum}{subcasenum}

\newcounter{stagenum}

\usepackage{graphicx}

\graphicspath{{counterfigs/}}

\usepackage{subfigure}

\newenvironment{selectionrule}
{

  \list{}{%
    \leftmargin0.5in   
    \rightmargin0cm 
    \setlength{\parindent}{-0.4in}
    \setlength{\itemindent}{-0.4in}
  }
  \item\relax
}
{
\endlist

}

\def\R{{\mathbb{R}}}
\def\Z{{\mathbb{Z}}}
\def\S{{\mathbb{S}}}

\def\Aut{\operatorname{Aut}}

\def\cO{{\mathcal O}}

\def\vx{{\mathbf x}}
\def\vb{{\mathbf b}}

\def\red{\operatorname{R}}
\def\blue{\operatorname{B}}
\newcommand{\taur}{\tau_{\red}}
\newcommand{\taub}{\tau_{\blue}}
\newcommand{\redreflect}[1]{\taur^{(#1)}}
\newcommand{\bluereflect}[1]{\taub^{(#1)}}
\newcommand{\rotation}[1]{\sigma_{#1}}

\def\im{{\rm im}}

\def\extD{\operatorname{ext}_D}

\title{Extending Precolorings to Distinguish Group Actions}
\author{Michael Ferrara$^{1,2}$ \and Ellen Gethner$^3$ \and Stephen G. Hartke$^{4}$ \and Derrick Stolee$^{5}$ \and Paul S. Wenger$^{6}$}

\begin{document}

\maketitle

\begin{abstract}

Given a group $\Gamma$ acting on a set $X$, a $k$-coloring $\phi:X\to\{1,\dots,k\}$ of $X$ is \textit{distinguishing} with respect to $\Gamma$ if the only $\gamma\in \Gamma$ that fixes $\phi$ is the identity action.  The \textit{distinguishing number} of the action $\Gamma$, denoted $D_{\Gamma}(X)$, is then the smallest positive integer $k$ such that there is a distinguishing $k$-coloring of $X$ with respect to $\Gamma$.  This notion has been studied in a number of settings, but by far the largest body of work has been concerned with finding the distinguishing number of the action of the automorphism group of a graph $G$ upon its vertex set, which is referred to as the distinguishing number of $G$.    

The distinguishing number of a group action is a measure of how difficult it is to ``break'' all of the permutations arising from that action.
In this paper, we aim to further differentiate the resilience of group actions with the same distinguishing number.  In particular, we introduce a precoloring extension framework to address this issue.  A set $S \subseteq X$ is a \emph{fixing set} for $\Gamma$ if for every non-identity element $\gamma \in \Gamma$ there is an element $s \in S$ such that $\gamma(s) \neq s$.  The \emph{distinguishing extension number} $\extD(X,\Gamma;k)$ is the minimum number $m$ such that for all fixing sets $W \subseteq X$ with $|W| \geq m$, every $k$-coloring $c : X \setminus W \to [k]$ can be extended to a $k$-coloring that distinguishes $X$.

In this paper, we prove that $\extD(\R,\Aut(\R),2) =4$, where $\Aut(\R)$ is comprised of compositions of translations and reflections.  We also consider the distinguishing extension number of the circle and (finite) cycles, obtaining several exact results and bounds.

{\bf Keywords:} Distinguishing Coloring, Distinguishing Number, Group Action.
\end{abstract}

\footnotetext[1]{Department of Mathematical and Statistical Sciences, University of Colorado Denver, Denver, CO 80217.\\ \texttt{michael.ferrara@ucdenver.edu}}
\footnotetext[2]{Research supported in part by Simons Foundation Grant \#206692.}
\footnotetext[3]{Department of Computer Science and Engineering, University of Colorado Denver, Denver, CO 80217.\\ \texttt{ellen.gethner@ucdenver.edu.}}
\footnotetext[4]{Department of Mathematics, University of Nebraska-Lincoln, Lincoln, NE 68588. \texttt{hartke@math.unl.edu}. Research supported in part by NSF Grant DMS-0914815.}
\footnotetext[5]{Department of Mathematics, Department of Computer Science, Iowa State University, Ames, IA 50011. \texttt{dstolee@iastate.edu}}
\footnotetext[6]{School of Mathematical Sciences, Rochester Institute of Technology, Rochester, NY 14623. \texttt{pswsma@rit.edu}}
\renewcommand{\thefootnote}{\arabic{footnote}}

\section{Introduction}

Given a group $\Gamma$ acting on a set $X$, a $k$-coloring $\phi:X\to\{1,\dots,k\}$ of $X$ is \textit{distinguishing} with respect to $\Gamma$ if the only $\gamma\in \Gamma$ that fixes $\phi$ is the identity action.  The \textit{distinguishing number} of the action $\Gamma$, denoted $D_{\Gamma}(X)$, is then the smallest positive integer $k$ such that there is a distinguishing $k$-colorings of $X$ with respect to $\Gamma$.  

The notion of distinguishing a general group action was introduced by Tymoczko in \cite{T04}, where a number of results on actions of $S_n$ appear, and was also addressed in \cite{chan2006distinguishing,chan2006maximum}. In \cite{KWZ06},  the distinguishing number of $GL_n(K)$ over a field $K$ acting on the vector space $K^n$ was completely determined.   We are most frequently concerned with the case where $\Gamma$ is the action of a symmetry group on some geometric or combinatorial object.  In particular, the overwhelming body of work on this problem is concerned with determining the \textit{distinguishing number} of a graph $G$, first introduced by Albertson and Collins \cite{AC} in 1996. 

Specifically, a vertex coloring of a graph $G$, $c : V(G) \to \{ 1,\dots,k \}$, is said to be \emph{distinguishing} if the only automorphism of $G$ that preserves all of the vertex colors is the identity.  The \emph{distinguishing number} of a graph $G$, denoted $D(G)$, is the minimum integer $r$ such that $G$ has a distinguishing $r$-coloring.  In the notation outlined above for general group actions, we therefore have that $D(G)=D_{Aut(G)}(V(G))$.  The distinguishing number of a graph has been widely studied for both finite (see  \cite{ABGeom06,AC,ChengPlanar08,ChengTree06,ChengInterval09}) and, starting in \cite{IWK07}, infinite (see \cite{BonDist10,BonHom12,laflamme2010distinguishing,LeInf13,WZInf07}) graphs. 

The distinguishing number of a group action is a measure of how difficult it is to ``break'' all of the permutations arising from that action;
the more colors required, the more resilient the action.
Almost all graphs have trivial automorphism group (see \cite{bollobas2004extremal}), and hence have distinguishing number 1. Many other familiar graph classes have distinguishing number 2 (see \cite{AlbCart05,AC,BCCubeDist04,FI08,IJK08,IK06}) despite their diverse collection of automorphism groups and seemingly disparate structural properties.  This leads us to ask the following:

\begin{center}
How can we further differentiate the resilience of group actions with the same distinguishing number?
\end{center}

In this paper, we propose a precoloring extension approach to this question.

\subsection{Extending Precolorings to Distinguishing Colorings}

A \emph{precoloring} of $X$ is a $k$-coloring of $X\setminus W$ for some $k$ and subset $W$ of $X$.  
We want to understand when a precoloring can be extended to a distinguishing $k$-coloring of all of $X$.
Specifically, given a set $W$ where every precoloring of $X \setminus W$ can be extended to a distinguishing coloring of $X$, it follows that an arbitrary $k$-coloring $c$ of $X$ can be modified into a distinguishing $k$-coloring by changing at most the colors on $W$.

A precoloring of $X\setminus W$ cannot be extended to a distinguishing coloring if there is a nontrivial element $\sigma \in \Gamma$ such that $\sigma$ pointwise stabilizes $W$; hence we preclude such subsets $W$ from our consideration.  
Formally, a set $S \subseteq X$ is a \emph{fixing set} for $\Gamma$ if the pointwise stabilizer of $S$ in $\Gamma$ is trivial; that is, for every non-identity element $\sigma \in \Gamma$ there is an element $s \in S$ such that $\sigma(s) \neq s$.  
We define the \emph{distinguishing extension number} $\extD(X,\Gamma;k)$ to be the minimum number $m$ such that for all fixing sets $W \subseteq X$ with $|W| \geq m$, every precoloring $c : X \setminus W \to \{1,\dots,k\}$ can be extended to a $k$-coloring of $X$ that is distinguishing under $\Gamma$.  
If $G$ is a graph, we write $\extD(G;k)$ instead of the more cumbersome $\extD(V(G),\Aut(G);k)$, and to further simplify, we write $\extD(G)$ instead of $\extD(G;D(G))$ and unambiguously refer to this quantity as the \textit{distinguishing extension number} of $G$.
We similarly define $\extD(X) = \extD(X, \Gamma;D_\Gamma(X))$ to be the \textit{distinguishing extension number} of $X$.

The problem of determining $\extD(X,\Gamma;k)$ can be viewed as a partizan combinatorial game that fits under the broad umbrella of competitive graph coloring.  Given a group $\Gamma$ acting on a set $X$, $m\ge 0$ and $k\ge D_{\Gamma}(X)$, two players, the Hero and and the Adversary, play the following game.  The Adversary begins by coloring all but an $m$-element fixing set of $X$ using colors from $\{1,\dots,k\}$.  The Hero wins if he can extend the Adversary's coloring to a distinguishing $k$-coloring of $X$ by coloring the $m$ uncolored elements using colors from $\{1,\dots,k\}$.  
The distinguishing extension number $\extD(X,\Gamma;k)$ is then the minimum $m$ such that the Hero has a guaranteed win.

Our work in this paper is further motivated (and contextualized) by the following problem in graph coloring: Given a graph $G$ and a $k$-coloring $c$ of some subset of $V(G)$, when can $c$ be extended to a {proper} $k$-coloring of $G$?  This problem was first introduced in \cite{PCI,PCII,PCIII}, and has been studied over the last twenty years not only in the context of proper colorings (see also \cite{AlbHutchPCplanar, albertson2004precoloring,  PruVoPC09,TuzaPCsurvey}), but also for list \cite{albertson98listextend, AxHutchListPC11}, circular \cite{AlbWestCircPC06,BrewNoelCircPC12} and fractional \cite{FracPC12} colorings of graphs.  As is the case here, the broad class of precoloring extension problems provide a framework by which it is possible to contrast the colorability of graphs that have the same value of a particular coloring parameter.

The remainder of this paper is structured as follows.  In Section~\ref{sec:prelim} we present some basic notions and state our main results.  In Section~\ref{sec:overview} we discuss an overview of our proof technique, with more detailed discussion of uncolored elements in Section~\ref{sec:uncolored} and with the final proofs of the main results in Section~\ref{sec:proof}.  We conclude with a discussion of future work in Section~\ref{sec:conclusion}.

%

\section{Preliminaries and Statement of Main Results}\label{sec:prelim}

In this paper, we study the distinguishing extension number of the real line and the unit circle.  Our investigation of the latter also allows us to naturally study the distinguishing extension number of the cycle $C_n$.  We begin by more generally considering $\R^n$, where $\Aut(\R^n)$ is the action of affine linear maps $\vx \mapsto A\vx + \vb$ where $A$ is a matrix with determinant in $\{+1,-1\}$.  For instance, the automorphisms of $\R$ are compositions of translations of $\R$ and reflections of $\R$ about a point.

Let $\S^d$ be the set of vectors $\vx\in\R^{d+1}$ with $\|\vx\|_2=1.$  The automorphisms of $\S^d$ are given by the $(d+1)\times (d+1)$ real matrices with determinant in $\{ +1, -1\}$.  When considering $\S^1$, we instead use the parameterization $\phi : \R \to \S^1$ given by $\phi(t) = (\cos(2\pi t), \sin(2\pi t))^\top$.  Hence, we  consider $\R/\Z$, the collection of preimages of $\phi$, to be the \emph{unit circle} with automorphisms given by rotations ($x \mapsto x + \alpha$) and reflections ($x \mapsto \beta - x$). 

We first consider $\extD(\R)$.  Coloring $(-\infty,0)$ red and $[0,\infty)$ blue is a distinguishing 2-coloring of $\R$, so we have that $D_{\Aut(\R)}(\R)=2$.  Next, we give a lower bound on $\extD(\R)$.

\begin{proposition}\label{prop:extD_R_LB}
$\extD(\R)\ge 4$.
\end{proposition}

\begin{proof}
Color $\R\setminus\{0,1,-1\}$ red and leave the remaining three elements blank.  We claim that this coloring cannot be extended to a distinguishing 2-coloring of $\R$.  Indeed, suppose that $c$ is an extension of this coloring that uses only colors red and blue.  If $c$ assigns all three elements of $\{0,1,-1\}$ the same color, then the reflection of $\R$ about 0 fixes $c$.  Further, if only one of the uncolored elements is colored blue, then the reflection of $\R$ about that point preserves $c$.  Hence we may assume that exactly two points in $\{-1,0,1\}$, call them $x$ and $y$, are colored blue.  It then follows that the reflection about $\frac{x+y}{2}$ preserves $c$.  This completes the proof.
\end{proof}

Our first main result shows that this lower bound is sharp.  

\begin{theorem}\label{thm:main1}
$\extD(\R) =4$.
\end{theorem}

For a graph $G$, an injection $\phi : V(G) \to X$ is a \emph{$\Gamma$-faithful embedding} of $G$ (into $X$) if the following properties hold.  First, there is an isomorphism $\varphi$ between the subgroup $\Gamma'$ of $\Gamma$ that setwise stabilizes $\im(\phi)$ and $\Aut(G)$, where $\im(\phi)$ denotes the image of $\phi$.   Second, it is possible to choose $\varphi$ such that for each $\gamma\in\Gamma'$, the permutation $\sigma_{\gamma}$ of $\im(\phi)$ corresponding to the action of $\gamma$ satisfies $$ \sigma_{\gamma}(\phi(v)) = \phi( \varphi(\gamma)(v))$$ 
for all $v\in V(G)$.  
That is, the  permutation $\sigma_\gamma$ of $\phi(V(G))$ induced by an action $\gamma \in \Gamma$ that stabilizes $\phi(V(G))$ is equal to the permutation of $\phi(V(G))$ induced by the corresponding automorphism $\varphi(\gamma)$ of $G$.
 Note that for any $n\ge 3$, any set of $n$ equally spaced points on the unit circle naturally corresponds to a faithful embedding of $C_n$ into $\S^1$.  
 Further, if $\ell$ divides $n$, $C_{\ell}$ has an $\Aut(C_n)$-faithful embedding into $V(C_n)$.  
 Finally, the unit cube $Q_3$ has a faithful embedding into $\S^2$.
 

The following lemmas formalize two ways in which faithful embeddings can be used to provide useful bounds on $\extD(X)$.    

\begin{lemma}\label{lemma:faithful}
	Let $\Gamma$ be a group acting on a set $X$ and $k \geq D_\Gamma(X)$. If $G$ is a graph with a $\Gamma$-faithful embedding into $X$ such that $D(G) \leq k$, then	
\[
	\extD(X,\Gamma;k) \geq \extD(G;k).
\]
\end{lemma}

\begin{proof}
	Let $\phi : V(G) \to X$ be a $\Gamma$-faithful embedding of $G$ into $X$, $\ell = \extD(G;k)$ and let $\Gamma'$ be the subgroup of $\Gamma$ that setwise stabilizes $X\setminus\im(\phi)$.  Since $\extD(G;k) \geq \ell$, let $Y \subseteq V(G)$ be a set of size $\ell-1$ such that there is a $k$-coloring $c : V(G)\setminus Y \to [k]$ which does not extend to a distinguishing $k$-coloring of $G$.  Color each element of $X\setminus\im(\phi)$ with color 1, and each element $x\in \im(\phi)\setminus Y$ with $c(\phi^{-1}(x))$.  It then follows, from the definition of $c$ and $\Gamma'$, that this coloring of $X\setminus Y$ cannot extend to a distinguishing $k$-coloring of $X$, and the result follows.  
\end{proof}

\begin{lemma}\label{lemma:faithful2}
	Let $\Gamma$ be a group acting on a set $X$ and $k \geq D_\Gamma(X)$.
	If  $G$ is a graph with  a  $\Gamma$-faithful embedding into $X$ such that $D(G) > k$, then	
\[
	\extD(X,\Gamma;k) \geq 1 + |G|.
\]
\end{lemma}

\begin{proof}
	As in Lemma \ref{lemma:faithful}, let $\phi : V(G) \to X$ be a $\Gamma$-faithful embedding of $G$ into $X$, and let $\Gamma'$ be the subgroup of $\Gamma$ that setwise stabilizes $X'=X\setminus \im(\phi)$.  If we color all of $X'$ using color 1, then it is not possible to extend this precoloring to a distinguishing $k$-coloring of $X$, as such a $k$-coloring on $\im(\phi)$ would induce a distinguishing $k$-coloring of $G$.
\end{proof}

The distinguishing number for graphs was introduced (in the guise of an entertaining problem) and determined for finite cycles in \cite{Rdist79}.
Of particular interest here is the observation that when $n \geq 6$ we can distinguish $C_n$ with two colors, but $D(C_3)=D(C_4)=D(C_5)=3$. 
Applying Lemma \ref{lemma:faithful2} with $C_3$, $C_4$, or $C_5$ gives rise to the following conjecture. 
\begin{conjecture}\label{conj:main}
\[
	\extD(C_n) = \begin{cases} 4 & \text{ if }n\not\equiv 0 \pmod {4 \text{ or }5}, \\ 
	5 & \text{ if }n\equiv 0 \pmod 4, \\
	6 & \text{ if }n\equiv 0 \pmod 5. \end{cases}
\]
\end{conjecture}

 The uncolored elements and precolorings in each diagram of Figure~\ref{fig:sharp} establish the sharpness of this conjecture.  Note that these precolorings also demonstrate that a non-extendible coloring need not use only one color.  

\begin{figure}[htp]
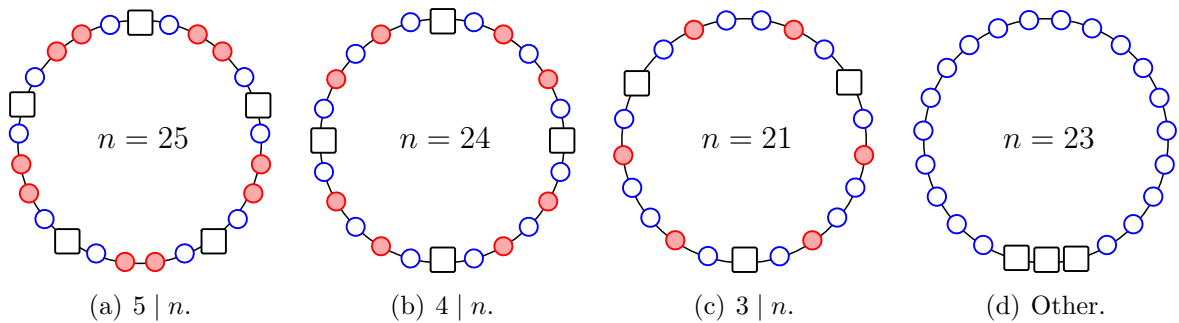

\centering
\mbox{
\subfigure[$5 \mid n$.]{
\begin{lpic}[]{"SharpExampleDivisor5"(36mm,)}
\lbl[]{24,24;$n=25$}
\end{lpic}
}
\subfigure[$4 \mid n$.]{
\begin{lpic}[]{"SharpExampleDivisor4"(36mm,)}
\lbl[]{24,24;$n=24$}
\end{lpic}
}
\subfigure[$3 \mid n$.]{
\begin{lpic}[]{"SharpExampleDivisor3"(36mm,)}
\lbl[]{24,24;$n=21$}
\end{lpic}
}
\subfigure[Other.]{
\begin{lpic}[]{"SharpExampleOther"(36mm,)}
\lbl[]{24,24;$n=23$}
\end{lpic}
}
}
\caption{\label{fig:sharp}Sharpness examples for Conjecture \ref{conj:main}.}
\end{figure}

We verify Conjecture \ref{conj:main} in infinitely many cases, based on the prime factorization of $n$.

\begin{theorem}\label{thm:main3}
If the minimum prime divisor of $n$ is at least $7$, then $\extD(C_n) = 4$.
\end{theorem}

We also show that for all $n$, both $\extD(\S^1)$ and $\extD(C_n)$ are bounded by an absolute constant. 

\begin{theorem}\label{thm:main2}
For all $n\ge 3$, $\extD(C_n) \leq \extD(\S^1) \leq 16$.
\end{theorem}

Note that the inequality  $\extD(C_n) \leq \extD(\S^1)$ follows from Lemma \ref{lemma:faithful}.  As $C_5$ has a faithful embedding into $C_{10}$, $\extD(C_{10}) \geq 6$ by Lemma \ref{lemma:faithful2}, and since $C_{10}$ has a faithful embedding into $\S^1$, $\extD(\S^1) \geq 6$ by Lemma \ref{lemma:faithful}.  We  conjecture that this lower bound is correct.

\begin{conjecture}\label{conj:S1}
$\extD(\S^1)=6$.
\end{conjecture}

\section{Overview of the Proof Technique}
\label{sec:overview}
\label{section:technique}

The proofs of Theorems \ref{thm:main1}, \ref{thm:main3}, and \ref{thm:main2} utilize a common argument, so we prove all three results simultaneously.
In this section we set up terminology and notation that is used throughout the rest of the paper.

Let $X$ be one of $C_n$, $\R/\Z$, or $\R$.
Each choice of $X$ has two categories of automorphisms: reflections and translations.
In the case of cycles and the unit circle, translations are rotations so we refer to rotations as translations for the sake of uniformity.
If a reflection $\tau$ stabilizes an element $x \in X$, then we say that $\tau$ is the \emph{reflection about $x$}.

Observe that every set of at least three elements in $X$ is a fixing set.
For a set $W \subset X$, let $c_0$ be a precoloring of $X \setminus W$.
We refer to the elements of $W$ as \textit{blanks}.
Let $\gamma\in \Aut(X)$.
We say that $c_0$ \textit{permits} $\gamma$ if there is an extension $c^*$ of $c_0$ to $X$ such that $\gamma$ preserves $c^*$; that is, $c^*(x)=c^*(\gamma(x))$ for all $x\in X$.

We will now give a brief outline of the proof.
Assume for the sake of contradiction that $c_0:X\setminus W\to \{\red,\blue\}$ is a red-blue coloring of $X\setminus W$ such that no extension of $c_0$ to $X$ is distinguishing.  First we will prove that there exists a point $w_0 \in W$ such that the reflection about $w_0$ sends the elements of $W\setminus\{w_0\}$ to elements outside of $W$.  Therefore, there is at most one extension of $c_0$ to $W\setminus\{w_0\}$ that permits the reflection about $w_0$.
Next we prove that there is at most one extension of $c_0$ to $W\setminus\{w_0\}$ that permits a translation.
Any extensions to $W\setminus\{w_0\}$ that permits either the reflection about $w_0$ or a translation are \textit{forbidden}; we show in all cases that there are at most two forbidden extensions of $c_0$ to $W\setminus\{w_0\}$.

Fix a non-forbidden extension of $c_0$ to $W\setminus\{w_0\}$; call it $c$.
Since no extension distinguishes $X$ and non-forbidden extensions do not permit translations, $c$ must permit a reflection.
Furthermore, since $c$ does not permit the reflection about $w_0$, no reflection permitted by $c$ fixes $w_0$.
Thus the reflection permitted when $c$ is extended by coloring $w_0$ red is distinct from the reflection permitted when $c$ is extended by coloring $w_0$ blue, since $w_0$ has distinct images under these reflections.
Let $\taur$ be the reflection permitted when $w_0$ is colored red and call it the {\it red reflection} permitted by $c$.
Let $\taub$ be the reflection permitted when $w_0$ is colored blue and call it the {\it blue reflection} permitted by $c$.
Since $\taur$ and $\taub$ are distinct, the composition $\taub\circ \taur$ yields a nontrivial translation $\sigma$ of $X$; we say that $\sigma$ is {\it generated} by $c$.
See Figure~\ref{fig:sigmatraversal} for examples of $\taur, \taub$, and $\sigma$ for the circle and the real line.

\begin{figure}[hp] 
\centering
\subfigure[Building $\sigma$ from $\tau_{\red}$ and $\tau_{\blue}$ for cycles and circles.]{
\scalebox{1}{
\begin{lpic}[]{"CircleTraversal"(95mm,)}
\lbl[]{83.5,150;\scriptsize $w_0$}
\lbl[r]{60,161;\footnotesize $\tau_{\red}$}
\lbl[l]{106,161;\footnotesize $\tau_{\blue}$}
\lbl[bl]{126,150;\footnotesize $\sigma$}
\end{lpic}
}
}
\vspace{0.5em}
\subfigure[Building $\sigma$ from $\tau_{\red}$ and $\tau_{\blue}$ for the real line.]{
\scalebox{1}{
\begin{lpic}[]{"RealLineTraversal"(95mm,)}

\lbl[]{66,75;\scriptsize $w_0$}
\lbl[tr]{50.5,140;\footnotesize $\tau_{\red}$}
\lbl[tl]{71,3;\footnotesize $\tau_{\blue}$}
\lbl[b]{76,90;\footnotesize $\sigma$}
\end{lpic}}
}
\caption{Building $\rotation{i}$.}
\end{figure}

We consider the orbits of the elements of $X$ under the actions of the group generated by $\sigma$.
In particular, we wish to understand the colors of the elements in these orbits.
Let $x\in X\setminus\{w_0,\taur(w_0)\}$.
Since $c$ permits $\taur$, it follows that $c(x) = c(\taur(x))$, so $\taur$ is color-preserving on $X\setminus\{w_0,\taur(w_0)\}$.
As $c$ is not defined on $w_0$, $\taur$ is not color-preserving on $\{w_0,\taur(w_0)\}$.
Similarly, $\taub$ is color-preserving on $X\setminus\{w_0,\taub(w_0)\}$.
Observe that 
\[
\begin{array}{cccccccc}
c(\sigma(w_0)) &=& c(\taub(\taur(w_0)) &=& c(\taur(w_0)) &=& \red\\
c(\sigma^{-1}(w_0)) &=& c(\taur(\taub(w_0)) &=& c(\taub(w_0)) &=& \blue
\end{array}
\]
hold for our definitions of $\taur$ and $\taub$.

Now let $x\in X\setminus\{w_0, \taur(w_0), \sigma^{-1}(w_0)\}$.
It follows that $c(x) = c(\taur(x))=c(\taub(\taur(x))) = c(\sigma(x))$.
The only elements whose image do not have the same color under $\sigma$ are: $\sigma^{-1}(w_0)$, whose image is a blank; $\taur(w_0)$, whose image is the blue element $\taub(w_0)$; and $w_0$, which is itself a blank.
Therefore $\sigma$ is, in a sense, nearly color-preserving.

All of these observations regarding the behavior of the permitted reflections of non-forbidden extensions are summarized in the following fact.

\clearpage
\begin{fact}\label{prop:sigmatraversal}
	Let $c$ be a non-forbidden extension of $c_0$ to $W\setminus\{w_0\}$, let $\taur$ and $\taub$ be the red and blue reflections permitted by $c$, respectively, and let $\sigma = \taub \circ \taur$.
	The following properties hold:
	\begin{itemize}
		\item[(F0)] $c(\sigma(w_0)) = \red$ and $c(\sigma^{-1}(w_0)) = \blue$; 
		\item[(F1)] $\sigma$ maps a red element to a blue element only at $\taur(w_0) \stackrel{\sigma}{\mapsto} \taub(w_0)$;
		\item[(F2)] $\sigma$ never maps a blue element to a red element, and $\sigma^2$ maps a blue element to a red element only at $\sigma^{-1}(w_0) \stackrel{\sigma}{\mapsto} w_0 \stackrel{\sigma}{\mapsto} \sigma(w_0)$;
		\item[(F3)]  $c(\sigma(x)) = c(x)$ for all elements $x \notin \{ \taur(w_0), \sigma^{-1}(w_0), w_0 \}$.
	\end{itemize}
\end{fact}

Figure \ref{fig:sigmatraversal} shows the possible colorings of the orbit of $\sigma$ containing $w_0$.
This orbit is either infinite or finite, and may or may not contain $\taur(w_0)$ and $\taub(w_0)$.
In particular,  when $X=\R$, the orbit is always infinite.
When $X = \R/\Z$, the orbit may be infinite or finite.
When $X = C_n$, the orbit is always finite.

\begin{figure}[tp]
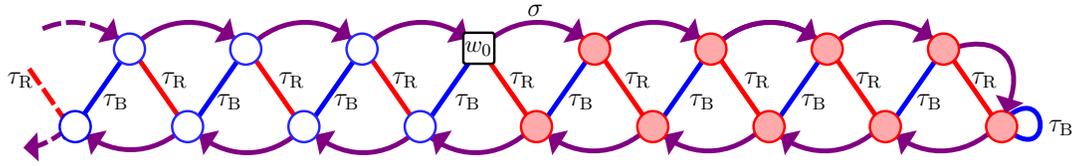
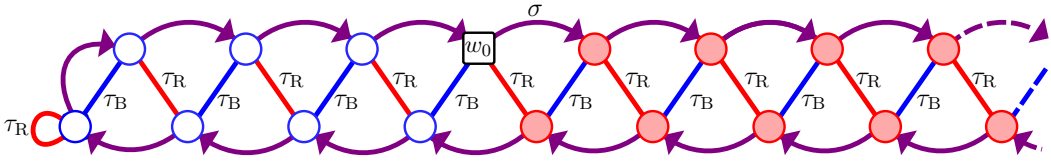
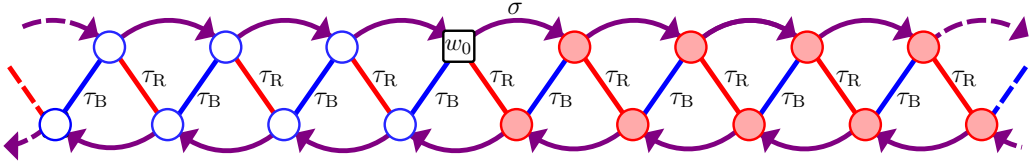
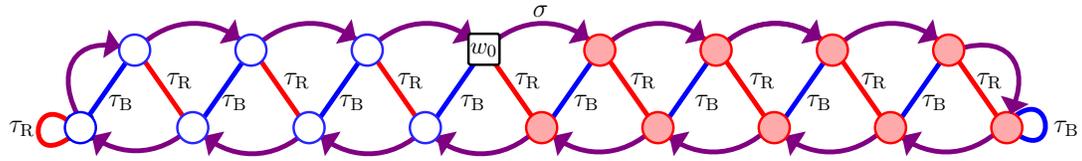

	\centering
	\subfigure[\label{subfig:sigmainfiniteA}\emph{Infinite Blue Case}, $|\sigma| = \infty$ and $\taur(w_0) = \sigma^k(w_0)$ for $k \geq 1$.]{\scalebox{0.825}{
\begin{lpic}[]{"SigmaTraversalInfA"(,25mm)}
\lbl[]{62.5,16.25;\small $w_0$}

\lbl[]{0,12; $\taur$}

\lbl[]{21,12; $\taur$}
\lbl[]{13,9; $\taub$}

\lbl[]{37,12; $\taur$}
\lbl[]{28.5,9; $\taub$}

\lbl[]{52.5,12; $\taur$}
\lbl[]{44.5,9; $\taub$}

\lbl[]{68.5,12; $\taur$}
\lbl[]{61,9; $\taub$}

\lbl[]{84.25,12; $\taur$}
\lbl[]{76.5,9; $\taub$}

\lbl[]{100.5,12; $\taur$}
\lbl[]{93,9; $\taub$}

\lbl[]{115.75,12; $\taur$}
\lbl[]{108.25,9; $\taub$}

\lbl[]{131.5,12; $\taur$}
\lbl[]{124,9; $\taub$}

\lbl[]{142,5.5; $\taub$}

\lbl[]{70.25,21.5; $\sigma$}
\end{lpic}
}}
\vspace{1em}
	\subfigure[\label{subfig:sigmainfiniteB}\emph{Infinite Red Case}, $|\sigma| = \infty$ and $\taur(w_0) = \sigma^{-k}(w_0)$ for $k \geq 1$.]{\scalebox{0.825}{
\begin{lpic}[]{"SigmaTraversalInfB"(,25mm)}
\lbl[]{62.5,16.25;\small $w_0$}

\lbl[]{-0.5,5.5; $\taur$}

\lbl[]{21,12; $\taur$}
\lbl[]{13,9; $\taub$}

\lbl[]{37,12; $\taur$}
\lbl[]{28.5,9; $\taub$}

\lbl[]{52.5,12; $\taur$}
\lbl[]{44.5,9; $\taub$}

\lbl[]{68.5,12; $\taur$}
\lbl[]{61,9; $\taub$}

\lbl[]{84.25,12; $\taur$}
\lbl[]{76.5,9; $\taub$}

\lbl[]{100.5,12; $\taur$}
\lbl[]{93,9; $\taub$}

\lbl[]{115.75,12; $\taur$}
\lbl[]{108.25,9; $\taub$}

\lbl[]{131.5,12; $\taur$}
\lbl[]{124,9; $\taub$}

\lbl[]{70.25,21.5; $\sigma$}
\end{lpic}
}}
\vspace{1em}
	\subfigure[\label{subfig:sigmainfiniteC}\emph{Doubly Infinite Case}, $|\sigma| = \infty$ and $\taur(w_0) \neq \sigma^{k}(w_0)$ for $k \in \Z$.]{\scalebox{0.825}{
\begin{lpic}[]{"SigmaTraversalInfC"(,25mm)}
\lbl[]{62.5,16.25;\small $w_0$}

\lbl[]{21,12; $\taur$}
\lbl[]{13,9; $\taub$}

\lbl[]{37,12; $\taur$}
\lbl[]{28.5,9; $\taub$}

\lbl[]{52.5,12; $\taur$}
\lbl[]{44.5,9; $\taub$}

\lbl[]{68.5,12; $\taur$}
\lbl[]{61,9; $\taub$}

\lbl[]{84.25,12; $\taur$}
\lbl[]{76.5,9; $\taub$}

\lbl[]{100.5,12; $\taur$}
\lbl[]{93,9; $\taub$}

\lbl[]{115.75,12; $\taur$}
\lbl[]{108.25,9; $\taub$}

\lbl[]{131.5,12; $\taur$}
\lbl[]{124,9; $\taub$}

\lbl[]{70.25,21.5; $\sigma$}
\end{lpic}
}}
\vspace{1em}
	\subfigure[\label{subfig:sigmafinite}\emph{Finite Case}, $|\sigma| < \infty$.]{
\scalebox{0.825}{
\begin{lpic}[]{"SigmaTraversal"(,25mm)}
\lbl[]{62.5,16.25;\small $w_0$}

\lbl[]{-0.5,5.5; $\taur$}

\lbl[]{21,12; $\taur$}
\lbl[]{13,9; $\taub$}

\lbl[]{37,12; $\taur$}
\lbl[]{28.5,9; $\taub$}

\lbl[]{52.5,12; $\taur$}
\lbl[]{44.5,9; $\taub$}

\lbl[]{68.5,12; $\taur$}
\lbl[]{61,9; $\taub$}

\lbl[]{84.25,12; $\taur$}
\lbl[]{76.5,9; $\taub$}

\lbl[]{100.5,12; $\taur$}
\lbl[]{93,9; $\taub$}

\lbl[]{115.75,12; $\taur$}
\lbl[]{108.25,9; $\taub$}

\lbl[]{131.5,12; $\taur$}
\lbl[]{124,9; $\taub$}

\lbl[]{142,5.5; $\taub$}

\lbl[]{70.25,21.5; $\sigma$}
\end{lpic}
}}
	\caption{\label{fig:sigmatraversal}The different traversals of $\sigma$.}
\end{figure}

Suppose that $c_1$ and $c_2$ are two non-forbidden extensions and assume that $c_1$ and $c_2$ do not extend to distinguishing colorings.
For each $i\in\{1,2\}$, let $\redreflect i$ and $\bluereflect i$ be the red and blue reflections permitted by $c_i$ and let $\rotation i$ be the rotation generated by $c_i$.
Both $\sigma_1$ and $\sigma_2$ satisfy Fact~\ref{prop:sigmatraversal} for their respective colorings.
We consider how $\rotation1$ and $\rotation2$ interact.
Some possible relationships between $\rotation1$ and $\rotation2$ are demonstrated in Figure~\ref{fig:lattices}.

Suppose that $\rotation1$ has infinite order.
If $\rotation{1}^{i} \neq \rotation{2}^j$ for all integers $i$ and $j$ (not both $0$), then $\sigma_2$ has infinite order and each $\rotation1$-orbit intersects each $\rotation2$-orbit in at most one element.
Thus the orbit of $w_0$ under the action of the group generated by $\sigma_1$ and $\sigma_2$ has a lattice structure, as in Figure~\ref{fig:lattice}.
Otherwise, there are integers $i,j$, not both $0$, such that $\rotation{1}^i = \rotation{2}^j$.
If $i=0$, then $\sigma_2$ has finite order, and the orbit of $w_0$ under $\sigma_1$ and $\sigma_2$ has a cylindrical structure.
If $|i|>0$ and $|j| \geq 2$, then $\sigma_2$ has infinite order and these actions have a shifted cylindrical structure as in Figure~\ref{fig:cylinder}.
When $\rotation{1}^i = \rotation2$ for some $i$, then we say that these actions form a \emph{linear lattice} as in Figure~\ref{fig:linearlattice}.
\begin{figure}[tp]
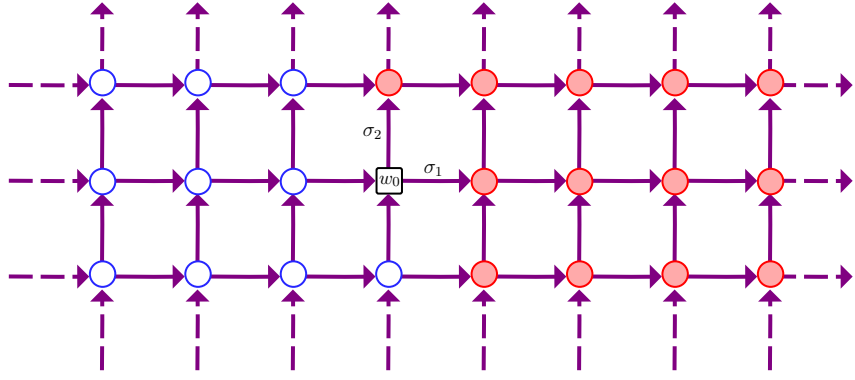
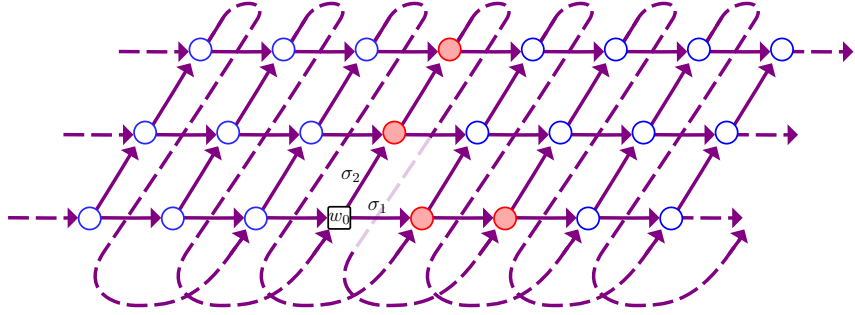
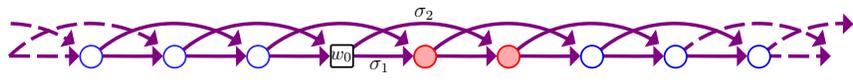
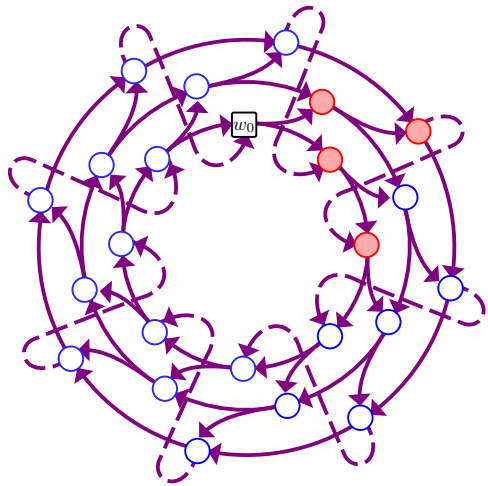
 
\centering
\subfigure[\label{fig:lattice}A lattice when $|\rotation1| = \infty$ and $\rotation{1}^i\neq\rotation{2}^j$ for all $i, j \neq 0$.]{
\scalebox{0.65}{
\begin{lpic}[]{"SigmaTraversal-Lattice"(175mm,)}
\lbl[]{64.25,32.5;\small $w_0$}
\lbl[b]{71.5,33.25; $\sigma_1$}
\lbl[r]{63,40.5; $\sigma_2$}
\end{lpic}
}
}
\subfigure[\label{fig:cylinder}A cylindrical lattice when $|\rotation1| = \infty$ and $\rotation{1}^2=\rotation{2}^3$.]{
\scalebox{0.65}{
\begin{lpic}[]{"SigmaTraversal-Cylinder"(175mm,)}
\lbl[]{64.25,18;\small $w_0$}
\lbl[b]{71.5,19; $\sigma_1$}
\lbl[br]{68.25,25; $\sigma_2$}
\end{lpic}
}
}
\subfigure[\label{fig:linearlattice}A linear lattice when $|\rotation1| = \infty$ and $\rotation{2}=\rotation{1}^2$.]{
\scalebox{0.65}{
\begin{lpic}[]{"SigmaTraversal-LinePair"(175mm,)}
\lbl[]{64.25,3.25;\small $w_0$}
\lbl[t]{71.5,2.5; $\sigma_1$}
\lbl[b]{80,10.25; $\sigma_2$}
\end{lpic}
}
}
\subfigure[\label{fig:torus}A toroidal lattice when $|\sigma_1| < \infty$.]{
\scalebox{0.65}{
\begin{lpic}[]{"SigmaTraversal-Torus"(100mm,)}
\lbl[]{42,63;\small $w_0$}
\end{lpic}
}
}
\caption{\label{fig:lattices}The translations $\sigma_1$ and $\sigma_2$ form lattices of different kinds.}
\end{figure}

If both $\rotation{1}$ and $\rotation{2}$ have finite order, then these actions define a torus as in Figure~\ref{fig:torus}.
In particular when $\rotation{1}^i = \rotation{2}$, we say that the actions form a {circular lattice}.
Observe that when $X = C_n$ where $n$ is prime, $\rotation1$ and $\rotation2$ automatically generate a circular lattice.

To obtain a contradiction, we choose $c_1$ and $c_2$ subject to a specified set of criteria.
We then prove that $\sigma_1$ and $\sigma_2$ cannot coexist with their given color-preserving properties.
This contradicts the assumption that every extension of $c_1$ or $c_2$ permits a color-preserving reflection.

\section{The Arrangement of Blanks}
\label{sec:uncolored}

In this section we will assume that the set $W$ satisfies some additional geometric conditions.
These conditions guarantee the existence of an element $w_0\in W$ such that the reflection about $w_0$ maps elements in $W\setminus\{w_0\}$ to elements not in $W$.
Our main theorems follow by demonstrating the existence of a subset of $W$ satisfying these special properties.

Let $X$ be the real line, the unit circle, or a cycle.
Every cycle has a faithful embedding into the circle, so we will use the image of such an embedding.
Thus every translation can be expressed as $x \mapsto x + \alpha$ for some real number $\alpha \in (0,1)$.
Given a set $W\subset X$ and a real number $\alpha$, define $W+\alpha=\{w+\alpha:w\in W\}$.

\begin{hypothesis}[Divisibility Condition]
Assume that $W$ satisfies $(W + \frac{i}{k}) \cap W = \varnothing$ for all $k \in \{2,3,4,5\}$ and all $i\in\{1\ldots,k-1\}$.
\end{hypothesis}

We will prove the following theorem.

\begin{theorem}\label{thm:divisibility}
If $W \subset X$ with $|W|  = 4$ satisfies the Divisibility Condition, then every precoloring $c: X \setminus W\to \{\red,\blue\}$ extends to a distinguishing 2-coloring of $X$.
\end{theorem}

Note that not every choice of $W$ will satisfy the divisibility condition.
However, we will show later in Lemma~\ref{lma:16} that  when $X$ is the unit circle or a cycle, if $W \subset X$ and $|W| \geq 16$, then there exists a 4-element subset of $W$  that satisfies the condition.
Therefore, Theorem~\ref{thm:divisibility} and Lemma~\ref{lma:16} imply Theorem~\ref{thm:main2}.

Observe that for any distinguishing coloring $c$ of $\R$ and any nonzero real number $\alpha$, the coloring $c'(x) = c(\alpha x)$ is also distinguishing.
Further note that if $W \subset \R$ with $|W| < \infty$ does not satisfy the divisibility condition, then  there exists a nonzero $\alpha$ such that the set $W' = \{ \alpha w : w \in W\}$ satisfies the divisibility condition.
Thus, if $c : \R \setminus W \to \{\red,\blue\}$ is a precoloring, then any distinguishing extension of the precoloring $c' : \R\setminus W' \to \{\red,\blue\}$ defined by $c'(x) = c(\alpha x)$ induces a distinguishing extension of $c$.
Therefore, Theorem~\ref{thm:divisibility} also implies Theorem~\ref{thm:main1}.

\begin{observation}\label{obs:2elts}
Let $W$ satisfy the Divisibility Condition and let $|W|\le 4$.
If $\sigma$ is a translation, then every $\sigma$-orbit contains at least one element not in $W$.
If $\sigma$ has order at least $3$, then every $\sigma$-orbit contains at least two elements not in $W$.
\end{observation}

\begin{lemma}\label{lma:divisibilityreflection}
If $W \subset X$ has $|W| = 4$ and satisfies the Divisibility Condition, then there exists an element $w_0 \in W$ such that the reflection $\tau_{w_0}$ about $w_0$ maps all elements in $W\setminus\{w_0\}$ to elements not in $W$.
\end{lemma}

\begin{proof}
If $X=\R$, then set $w_0=\min W$.
If $X=C_n$, we may assume that $W\subset \R/\Z$ by using the faithful embedding of $C_n$ into $\R/\Z$, and so we assume that $X=\R/\Z$.
Let $W=\{x_1,x_2,x_3,x_4\}$ labeled in clockwise order, and without loss of generality assume that $x_1=0$.
Assume for the purposes of obtaining a contradiction that $\tau_{x_i}$ maps an element of $W\setminus\{x_i\}$ to an element of $W\setminus\{x_i\}$ for all $x_i\in W$.

If $\tau_{x_1}:x_2\to x_4$, then $x_2=\alpha$ and $x_4=1-\alpha$ for some $\alpha\in (0,1/2)$.
Since $\alpha\neq 1/3$, we conclude that $\tau_{x_2}(x_1)\neq x_4$.
Therefore $\tau_{x_2}(x_3)\in \{x_1,x_4\}$.
By symmetry $\tau_{x_4}(x_3)\in \{x_1,x_2\}$.
Therefore $x_3\in\{2\alpha,3\alpha\}\cap\{1-2\alpha,1-3\alpha\}$.
Since $1=x_3+(1-x_3)$, we conclude that $1\in\{4\alpha,5\alpha,6\alpha\}$.
By the Divisibility Condition, we conclude that $\alpha=1/6$.
However, in this case, $x_3=1/2$, contradicting the Divisibility Condition.
Therefore we may assume that $\tau_{x_i}(x_{i+1})\neq x_{i-1}$ for all $i\in [4]$ with indices taken modulo $4$.

Now, without loss of generality, assume that $\tau_{x_1}:x_2\to x_3$.
Thus $x_2=\alpha$ and $x_3=1-\alpha$ for some $\alpha\in (0,1/2)$.
By the Divisibility Condition, $\alpha\neq 1/3$, so $\tau_{x_3}(x_2)\neq x_1$.
Since $x_3\in (1/2,1)$ and $x_4\in (x_3,1)$ it follows that $\tau_{x_3}(x_4)\neq x_1$.
Therefore $\tau_{x_3}(x_2)=x_4$, and by the previous paragraph, we reach a contradiction.
\end{proof}

Let $c : X \setminus W \to \{\red, \blue\}$ be a precoloring of $X \setminus W$ and assume that $c$ does not extend to a distinguishing coloring of $X$.
Let $c_1$ be an extension of $c$ to $X\setminus\{w_0\}$.

\begin{lemma}\label{lma:nonforbidden}
Let $W \subset X$ satisfy the Divisibility Condition with $|W|=4$.
Let $w_0$ be an element in $W$ such that $\tau_{w_0}$ maps elements in $W\setminus\{w_0\}$ to elements not in $W$.
There exist at most two forbidden extensions of $c$ to $X\setminus\{w_0\}$.
\end{lemma}

\begin{proof}
Since $\tau_{w_0}$ sends elements in $W\setminus\{w_0\}$ to elements not in $W$, the only extension that might permit $\tau_{w_0}$ is the coloring $c(w) = c(\tau_{w_0}(w))$ for all $w \in W- w_0$.

Assume that $c_1$ and $c_2$ are two distinct extensions of $c$ that permit translations $\gamma_1$ and $\gamma_2$, respectively, defined by $\gamma_1:x\mapsto x+\alpha_1$ and $\gamma_2:x\mapsto x+\alpha_2$ for some nonzero real numbers $\alpha_1$ and $\alpha_2$. 
Let $w\in W$ such that $c_1(w)\neq c_2(w)$.
For each $i \in \{1,2\}$, define $\cO_w^i$ to be $\{\gamma_i^j(w):j\in\Z\}$, the orbit of $w$ under $\gamma_i$.
If $\cO_w^1 = \cO_w^2$ then $\cO_w^1\subset W$, since the color of $\cO_w^1$ under $c_1$ must disagree with the color of $\cO_w^2$ under $c_2$.
This follows because $\gamma_1$ preserves $c_1$, so $c_1(w)=c_1(x)$ for all $x\in \cO_w$, and $\gamma_1$ preserves $c_2$, so $c_2(w)=c_2(x)$ for all $x\in \cO_w$.
Therefore $2\leq |\cO_w|\le 4$, violating the Divisibility Condition.
Thus, if $\gamma_1$ and $\gamma_2$ have finite orders (hence $X\neq \R$), then they have distinct orders.

Suppose that $\gamma_1$ has order $2$ and $\gamma_2$ has even order.
Therefore $w+\frac12$ is in both the $\gamma_1$-orbit and $\gamma_2$-orbit of $w$, and by the Divisibility Condition $w+\frac12\notin W$.
Therefore $c_1(w)=c_1(w+\frac12)=c_2(w+\frac12)=c_2(w)$, a contradiction.

In the remaining cases, we assume without loss of generality that $\gamma_1$ has order at least $3$ and $\gamma_2$ either has order $2$ or order greater than the order of $\gamma_1$.
Thus, for all  $\gamma_1$-orbits $\mathcal O$, we have that $|\mathcal O|\ge 3$, $\cO$ is monochromatic in $c_1$, and $\gamma_2(\mathcal O)\cap \mathcal O=\varnothing$.
If the order of $\gamma_1$ is at most $5$, then $|\mathcal O\cap W|\le 1$ and $|\gamma_2(\mathcal O)\cap W|\le 1$ by the Divisibility Condition.
Therefore there is $x\in \mathcal O$ such that $x$ and $\gamma_2(x)$ are not in $W$, and both $\mathcal O$ and $\gamma_2(\mathcal O)$ are the same color under $c_1$.
If $\gamma_1$ has order at least $6$, then there are at least six pairs of the form $(x,\gamma_2(x))$ where $x\in \mathcal O$, so there is $x\in \mathcal O$ such that $x$ and $\gamma_2(x)$ are not in $W$.
Therefore $\mathcal O$ and $\gamma_2(\mathcal O)$ are the same color in under $c_1$.
It then follows that for all integers $k$, the orbits $\mathcal O$ and $\gamma_2^k(\mathcal O)$ are the same color under $c_1$.

By Observation~\ref{obs:2elts}, there is a value $\ell_2$ such that $\gamma_2^{\ell_2}(w)\notin W$.
Again, let $\cO_w$ be the $\gamma_1$-orbit of $w$, and note that $\cO_w$ and $\gamma_2^{\ell_2}(\cO_w)$ have the same color under $c_1$.
Therefore,
$$c_1(w)=c_1(\gamma_2^{\ell_2}(w))=c_2(\gamma_2^{\ell_2}(w))=c_2(w),$$
a contradiction.
\end{proof}

\section{Proofs of the Main Theorems}
\label{sec:proof}

In this section, we will complete the proofs of Theorem~\ref{thm:main1}, \ref{thm:main3}, and \ref{thm:main2} by proving Theorem~\ref{thm:divisibility}.

Let $W$ be a set of four elements in $X$ satisfying the Divisibility Condition, and let $w_0\in W$ be the element guaranteed by Lemma~\ref{lma:divisibilityreflection}.
Let $c : X \setminus W \to \{\red,\blue\}$ be a precoloring of $X \setminus W$.
We use a sequence of extremal choices to extend $c$ to a coloring $c_1$ of $X\setminus\{w_0\}$.
We then select a subset $W' \subseteq W\setminus\{w_0\}$ and obtain the coloring $c_2$ by changing the colors of the elements in $W'$.
Our choices will guarantee that both $c_1$ and $c_2$ are not the (at most) two forbidden colorings guaranteed by Lemma~\ref{lma:nonforbidden}.
By assumption, neither $c_1$ nor $c_2$ extends to a distinguishing coloring, and we then use Fact~\ref{prop:sigmatraversal} to derive a contradiction.

Suppose that $c_1$ is a non-forbidden extension of $c$ to $X \setminus \{w_0\}$.
Let $\cO_{0}$ be the orbit of $w_0$ under the rotation $\sigma_1$ generated by $c_1$.
Let $\cO_0^{\red}$ be the set of elements in $\cO_0$ colored red by $c_1$ and let $\cO_0^{\blue}$ be the set of elements in $\cO_0$ colored blue by $c_1$.

\begin{hypothesis}[Selection criteria for $c_1$]
Among the non-forbidden extensions of $c$ to $X\setminus\{w_0\}$, select $c_1$ so that the following conditions are satisfied:

\begin{selectionrule}
(C0)  $|\cO_0|$ is maximum.
\end{selectionrule}

\begin{selectionrule}
(C1) Subject to (C0), if $c'$ is a nonforbidden extension of $c$ to $X\setminus\{w_0\}$ and
 $\cO_{c'}$ is the orbit of $w_0$ under the action of the rotation generated by $c'$, then $\cO_0\not\subset\cO_{c'}$.
\end{selectionrule}

\begin{selectionrule}
(C2) Subject to (C0) and (C1), $\min\{|\cO_0^{\red}|,|\cO_0^{\blue}|\}$ is maximum.
\end{selectionrule}


\end{hypothesis}

Note that these conditions are satisfied by at least one coloring since there are a finite number of extensions of $c$ to $X \setminus \{w_0\}$.
We now discuss some structural relationships between the rotation $\sigma_1$ and the locations of the elements of $W$ subject to the selection hypothesis.

\begin{observation}\label{obs:notbothsigmas}
Since the reflection about $w_0$ sends all elements of $W \setminus \{w_0\}$ to elements not in $W$, we have that
\begin{enumerate}
\item there is at most one element of $W$ among $\sigma_1(w_0)$ and $\sigma_1^{-1}(w_0)$, and
\item there is at most one element of $W$ among $\sigma_1(w_0)+\frac{1}{2}$ and $\sigma_1^{-1}(w_0)+\frac{1}{2}$.
\end{enumerate}
\end{observation}

\begin{lemma}\label{lma:notbothtaus}
If $\min\{|\cO_0^{\red}|,|\cO_0^{\blue}|\} < \infty$, then  $\taur^{(1)}(w_0)$  and $\taub^{(1)}(w_0)$ are not both in $W$. 
In particular, $\taur^{(1)}(w_0)\notin W$ if $|\cO_0^{\red}| < |\cO_0^{\blue}|-1$  and $\taub^{(1)}(w_0)\notin W$ if $|\cO_0^{\blue}| < |\cO_0^{\red}|-1$.
\end{lemma}
\begin{proof}
If $|\cO_0^{\red}| = |\cO_0^{\blue}| < \infty$, then  $\tau_{w_0}(\taur^{(1)}(w_0)) = \taub^{(1)}(w_0)$, and the statement follows by the definition of $w_0$.
If $\left||\cO_0^{\red}| - |\cO_0^{\blue}|\right|  = 1$, then $X\neq \R$ and one of $\taur^{(1)}(w_0)$ or $\taub^{(1)}(w_0)$ is equal to $w_0 + \frac{1}{2}$.
The element  $w_0 + \frac{1}{2}$ is not a blank by the Divisibility Condition.

Otherwise, suppose that $|\cO_0^{\red}| < |\cO_0^{\blue}| - 1$ and hence $|\cO_0^{\red}| < \infty$; the case when $|\cO_0^{\blue}| < |\cO_0^{\red}| - 1$ follows by a similar argument.
If $\taub^{(1)}(w_0) \in W$, then let $c'$ be the coloring that matches $c_1$ with the exception that $c'(\taub^{(1)}(w_0))=\red$.
Observe that $\cO_0$ is the only $\sigma_1$-orbit that is nonmonochromatic  under $c'$.
Let $\tau'$ be a reflection and $\sigma'$ be a rotation.
Note that $\tau'(\cO_0)$ and $\sigma'(\cO_0)$ are both $\sigma_1$-orbits.
Since at least half of the elements in $\cO_0$ are blue under $c'$, the $\sigma'$-orbit of $\sigma_1(w_0)$ contains a blue element under $c'$, and hence $c'$ does not permit $\sigma'$.
Since $c'(\sigma_1(w_0))\neq c'(\sigma_1^{-1}(w_0))$, $c'$ does not permit $\tau_{w_0}$ and thus $c'$ is not forbidden.

Therefore, the two extensions of $c'$ to $w_0$ only permit  reflections, call them $\taur'$ and $\taub'$ when $c'(w_0) = \red$ and $c'(w_0) = \blue$, respectively.
Observe that $\taur'(w_0)=\taub^{(1)}(w_0)$ and $\taub'(w_0)=\sigma_1(\taub^{(1)}(w_0))$.
Thus $\taub'\circ \taur'=\sigma_1$, so $\sigma_1$ is the rotation generated by $c'$.
However, $\cO_0$ has more red elements under $c'$ than $c_1$, contradicting (C2) in the selection criteria for $c_1$.
\end{proof}

We define a set $W' \subseteq W\setminus\{w_0\}$ depending on the number of elements of each color in $\cO_0$ as follows.

\begin{hypothesis}[Selection of $W'$]\quad
\begin{enumerate}
\item If $\{\taur^{(1)}(w_0),\taub^{(1)}(w_0)\} \subseteq W$, then let $W' = \{\taur^{(1)}(w_0),\taub^{(1)}(w_0)\}$.
\item If $\{\taur^{(1)}(w_0),\taub^{(1)}(w_0)\} \nsubseteq W$ and $\min\{ |\cO_0^{\red}|, |\cO_0^{\blue}|\} > 1$, then select 
\[
	w' \in W \setminus \{ w_0, \sigma_1(w_0), \sigma_1^{-1}(w_0), \taur^{(1)}(w_0), \taub^{(1)}(w_0)\},
\]
and let $W' = \{w'\}$.

\item If $\{\taur^{(1)}(w_0),\taub^{(1)}(w_0)\} \nsubseteq W$ and $\min\{ |\cO_0^{\red}|, |\cO_0^{\blue}|\} = 1$, then select 
\[
	w' \in W \setminus \left\{ w_0, \sigma_1(w_0), \sigma_1^{-1}(w_0), \taur^{(1)}(w_0), \taub^{(1)}(w_0), \sigma_1(w_0)+\frac{1}{2}, \sigma_1^{-1}(w_0)+\frac{1}{2}\right\},
\] 
and let $W' = \{w'\}$.
\end{enumerate}
%
\end{hypothesis}

We claim that the above choice is well-defined.
If $\{\taur^{(1)}(w_0),\taub^{(1)}(w_0)\}$ is not a subset of $W$ and $\min\{ |\cO_0^{\red}|, |\cO_0^{\blue}|\} > 1$, then such an element $w'$ exists by Observation~\ref{obs:notbothsigmas} and Lemma~\ref{lma:notbothtaus}.
If $\min\{ |\cO_0^{\red}|, |\cO_0^{\blue}|\} = 1$, then either $\sigma_1(w_0) = \taur^{(1)}(w_0)$ or $\sigma_1^{-1}(w_0) = \taub^{(1)}(w_0)$, and by Lemma~\ref{lma:notbothtaus} there is at most one blank among $\{\sigma_1(w_0), \taur^{(1)}(w_0), \taub^{(1)}(w_0), \sigma_1^{-1}(w_0)\}$.
Furthermore, Observation~\ref{obs:notbothsigmas} implies there is at most one blank among $\sigma_1(w_0)+\frac{1}{2}$ and $\sigma_1^{-1}(w_0)+\frac{1}{2}$, and hence the choice for $w'$ is possible.

Define $c_2$ to be the coloring obtained from $c_1$ by changing the colors of the elements in $W'$.
We will show that $c_2$ is not forbidden.

\begin{lemma}\label{lem:c_2notforbA}
The coloring $c_2$ does not permit the translation of order 2.
\end{lemma}

\begin{proof}
Let $\gamma$ be the translation of order $2$, and hence $X \neq \R$.
Suppose that $c_2$ permits $\gamma$.
Fix $c_2(w_0)$ such that $\gamma$ preserves $c_2$; for this proof, let $c_1(w_0) = c_2(w_0)$.
Consider cases based on how $W'$ was selected.

Suppose $W' = \{\taur^{(1)}(w_0), \taub^{(1)}(w_0)\}$, which implies that $|\cO_0^{\red}| = |\cO_0^{\blue}| = \infty$ by Lemma~\ref{lma:notbothtaus}.
Therefore, $\cO_0$ and the $\sigma_1$-orbit containing $\taur^{(1)}(w_0)$ are the two nonmonochromatic $\sigma_1$-orbits under $c_2$.
Since $|\cO_0| = \infty$, we have that $\gamma(\cO_0)$ is a distinct $\sigma_1$-orbit and hence $\gamma(\taur^{(1)}(w_0)), \gamma(\taub^{(1)}(w_0)) \in \cO_0$.
Recall $\sigma_1(\taur^{(1)}(w_0)) = \taub^{(1)}(w_0)$.
Hence $c_2(\sigma_1^{-1}(\taur^{(1)}(w_0))) = \red$, $c_2(\taur^{(1)}(w_0))=\blue$, and $c_2(\sigma_1(\taur^{(1)}(w_0)))=\red$.
There does not exist an integer $i$ such that $c_2(\sigma_1^{i-1}(w_0)) = \red$,  $c_2(\sigma_1^{i}(w_0)) = \blue$,  $c_2(\sigma_1^{i+1}(w_0)) = \red$, and hence $c_2$ does not permit $\gamma$.

We may now assume that $\{\taur^{(1)}(w_0),\taub^{(1)}(w_0)\} \nsubseteq W$ and $W' = \{ w'\}$.
Observe that $c_1(\sigma_1^{-1}(w')) = c_1(w')= c_1(\sigma_1(w'))$, but $c_2(\sigma_1^{-1}(w'))  \neq c_2(w')$ and $c_2(w') \neq c_2(\sigma_1(w'))$.
Let $x = \gamma(w') = w' + \frac{1}{2}$.
Since $|\sigma_1| \geq 3$, we conclude that $w' \notin \{ \sigma_1^{-1}(x), x, \sigma_1(x)\}$ and $c_1(y) = c_2(y)$ for all $y \in \{ \sigma_1^{-1}(x), x, \sigma_1(x)\}$.
Therefore, $c_1(\sigma_1^{-1}(x))  \neq c_1(x)$ and $c_1(x) \neq c_1(\sigma_1(x))$.
This implies that $x \in \cO_0$ and $x$ is the only element in $\cO_0$ of color $c_1(x)$ under $c_1$.
Hence $\min\{ |\cO_0^{\red}|, |\cO_0^{\blue}|\} = 1$ and $\gamma(w') \in \{ \sigma_1(w_0), \sigma_1^{-1}(w_0)\}$.
But this implies that $w' \in \left\{\sigma_1(w_0) + \frac{1}{2}, \sigma_1^{-1}(w_0) + \frac{1}{2}\right\}$, contradicting our choice of $w'$.
\end{proof}


The following technical lemma will be used extensively.

\begin{lemma}\label{cl:gammaorder}\label{lma:boundarypts}
Let $\gamma$ be a translation such that
\begin{enumerate}
\item $\gamma$ has order at least 3, 
\item the $\gamma$-orbits of $\sigma_1^{-1}(w_0)$, $w_0$, and  $\sigma_1(w_0)$ are distinct, and
\item $\gamma$ preserves $c_2$ on the $\gamma$-orbits of $\sigma_1^{-1}(w_0)$ and $\sigma_1(w_0)$.
\end{enumerate}
For $j \in \{-1,0,1\}$, let $S_j$ be the $\gamma$-orbit of $\sigma_1^j(w_0)$.
If $1\le i< |\gamma|$, then
\begin{enumerate}
\item $|\{\gamma^i(\sigma_1(w_0)),\gamma^i(\sigma_1^{-1}(w_0))\}\cap W'|\ge 1$,
\item $|W'| \ge |\gamma|-1$,  
\item $\gamma$ has order exactly 3, 
\item $W' \subset S_{-1} \cup S_1$, and
\item $W' = \{\taur^{(1)}(w_0), \taub^{(1)}(w_0)\}$.
\end{enumerate}
\end{lemma}

\begin{proof}
If $\{\gamma^i(\sigma_1(w_0)),\gamma^i(\sigma_1^{-1}(w_0))\}\cap W' = \varnothing$ for some $i \in [|\gamma|-1]$, then $c_1(\gamma^i(\sigma_1^{-1}(w_0)))=\blue$ and $c_1(\gamma^i(\sigma_1(w_0)))=\red$.
Therefore, $\sigma_1^2$ maps the blue element $\gamma^i(\sigma_1^{-1}(w_0))$ to the red element $\gamma^i(\sigma_1(w_0))$ under $c_1$.
By Fact~\ref{prop:sigmatraversal}, we conclude that $\gamma^i(\sigma_1^{-1}(w_0))=\sigma_1^{-1}(w_0)$ and $\gamma^i(\sigma_1(w_0))=\sigma_1(w_0)$, contradicting the assumption that $1\le i \le |\gamma|-1$.
Therefore $c_1(\gamma^i(\sigma_1(w_0)))=\blue$ or $c_1(\gamma^i(\sigma_1^{-1}(w_0)))=\red$, and consequently at least one of these elements must be in $W'$, proving conclusion 1.
Conclusions 2--4 follow directly from conclusion 1 and Lemma~\ref{lem:c_2notforbA}.
Conclusion 5 follows since $\{\taur^{(1)}(w_0), \taub^{(1)}(w_0)\}$ is the only choice of $W'$ that has two elements.
\end{proof}

\begin{lemma}\label{lem:c_2notforb}
$c_2$ is not forbidden.
\end{lemma}

\begin{proof}
The reflection about $w_0$ maps the red element $\sigma_1(w_0)$ to the blue element $\sigma_1^{-1}(w_0)$, and hence $c_2$ does not permit $\tau_{w_0}$.
Suppose $c_2$ permits a translation $\gamma$.
By Lemma~\ref{lem:c_2notforbA}, $\gamma$ has order at least 3.
Since $\gamma$ preserves $c_2$ and $\sigma_1(w_0), \sigma_1^{-1}(w_0) \notin W'$, the $\gamma$-orbits of $\sigma_1(w_0)$ and $\sigma_1^{-1}(w_0)$ are distinct, and hence also are distinct from the $\gamma$-orbit of $w_0$.
Therefore Lemma~\ref{cl:gammaorder} applies to $\gamma$, so  $\gamma$ has order exactly 3 and $W' = \{ \taur^{(1)}(w_0), \taub^{(1)}(w_0)\}$.

If $\gamma(\cO_0)\neq\cO_0$, then $\cO_0, \gamma(\cO_0), \gamma^2(\cO_0)$ are distinct.
Since there are at most two nonmonochromatic $\sigma_1$-orbits under $c_1$,  we assume  without loss of generality that $\gamma(\cO_0)$ is monochromatic under $c_1$ (since $c_2$ also permits $\gamma^{-1}$).
However, $W' \cap \gamma(\cO_0) = \varnothing$ and therefore either the $\gamma$-orbit of $\sigma_1(w_0)$ or the $\gamma$-orbit of $\sigma_1^{-1}(w_0)$ is not monochromatic under $c_2$, a contradiction.
\end{proof}

Let $\sigma_2$ be the rotation generated by $c_2$ (as in Fact~\ref{prop:sigmatraversal}).
We will demonstrate that $\sigma_2$ does not satisfy Fact~\ref{prop:sigmatraversal} under $c_2$.

\begin{lemma}\label{lma:sig2finite}
$|\rotation2| < \infty$.
\end{lemma}

\begin{proof}
Suppose $|\rotation2| = \infty$. 
By (C0), we also have that $|\sigma_1| = \infty$.
For $i \in \Z$, let $S_i$ be the $\sigma_2$-orbit of $\sigma_1^i(w_0)$.
For any $w \in W'$, $c_2(w) \neq c_2(\sigma_1(w))$ and hence $\sigma_2 \neq \sigma_1$.
Since $\sigma_1^{-1}(w_0) \notin W'$, we have that $\sigma_2 \neq \sigma_1^{-1}$.
Therefore, $S_0 \notin \{ S_1, S_{-1}\}$.

Suppose that $S_1 \neq S_{-1}$ and hence at least one of $S_1$ or $S_{-1}$ is monochromatic under $c_2$.
If $\min\{ |\cO_0^{\red}|, |\cO_0^{\blue}| \}= \infty$, then $|S_1 \cap \cO_0| = |S_{-1} \cap \cO_0| = 1$ and hence for some $\ell \geq 1$, $\sigma_2^\ell(\cO_0)$ is monochromatic under $c_2$.
However, there are now an infinite number of integers $k > 0$ such that $c_2(\sigma_1^{k}(w_0)) = \red$ and $c_2(\sigma_1^{-k}(w_0)) = \blue$, so $c_2(\sigma_2^{\ell}(\sigma_1^k(w_0))) \neq c_2(\sigma_2^{\ell}(\sigma_1^{-k}(w_0)))$, a contradiction.
Thus $\min\{ |\cO_0^{\red}|, |\cO_0^{\blue}| \} < \infty$ and therefore $S_0$ is the only nonmonochromatic $\sigma_2$-orbit under $c_2$.
This implies that $S_1$ and $S_{-1}$ are monochromatic under $c_2$, so Lemma~\ref{cl:gammaorder} applies and $|\sigma_2| = 3$.

Thus $S_1 = S_{-1}$. 
Since $\sigma_1(w_0), \sigma_1^{-1}(w_0) \in S_1$, then $S_1$ is not monochromatic and thus there are two nonmonochromatic $\sigma_2$-orbits under $c_2$.
Hence there are an infinite number of red elements and an infinite number of blue elements in both $S_0$ and $S_1$ under $c_2$.
Thus by (C0)--(C2), $|\cO_0^{\red}| = |\cO_0^{\blue}|=\infty$.
Let $i \in \Z$ be such that $\sigma_2^i(\sigma_1(w_0)) = \sigma_1^{-1}(w_0)$; note $\sigma_2^i = \sigma_1^{-2}$ and $i > 0$ by Fact~\ref{prop:sigmatraversal}.
Also by Fact~\ref{prop:sigmatraversal}, $c_2(\sigma_2^{-\ell}(w_0))=\blue$ for all $\ell \geq 1$.
However, there exists an integer $k > 0$ such that $\sigma_2^{-ik}(w_0) \notin W'$ and hence for $\ell = ik$ we have
\[
	\blue = c_2(\sigma_2^{-\ell}(w_0)) = c_2(\sigma_2^{-ik}(w_0)) = c_1(\sigma_2^{-ik}(w_0)) = c_1(\sigma_1^{2k}(w_0)) = \red,
\]
a contradiction.
\end{proof}

Theorem~\ref{thm:main1} now follows directly from Lemma~\ref{lma:sig2finite}.

\vspace{0.5em}
\noindent\textbf{Theorem~\ref{thm:main1}.} $\extD(\R) = 4$.

\begin{proof}[Proof of Theorem~\ref{thm:main1}.]
Let $W \subset \R$ be a set of size four, and let $w_0 = \min W$.
Thus $\tau_{w_0}$ sends all elements of $W\setminus\{w_0\}$ to elements not in $W$.
Let $c : \R \setminus W \to \{\red,\blue\}$ be a precoloring of $\R \setminus W$, and let $c_1$ and $c_2$ be the extensions of $c$ defined by the selection criteria (C0)--(C2) and the definition of $W'$.
The coloring $c_1$ is not forbidden by choice, and $c_2$ is not forbidden by Lemma~\ref{lem:c_2notforb}.
Lemma~\ref{lma:sig2finite} implies that $\sigma_2$, the translation generated by $c_2$, satisfies $|\sigma_2|<\infty$.
However, all translations of $\R$ are of infinite order, a contradiction.
Therefore there is a distinguishing extension of $c$.
%
\end{proof}

For the remainder of the paper, we assume that $X\neq \R$.
First we will show that if $|\sigma_1|=\infty$, then there is a distinguishing extension of $c$ to $X$.

%


\begin{lemma}\label{lma:infinitefinite}
 $|\rotation1| < \infty$.
\end{lemma}

\begin{proof}
Suppose otherwise that $|\rotation1| = \infty$.
Observe that $\rotation{1}^i = \rotation{2}^j$ if and only if $i = 0$ and $|\rotation{2}|$ divides $j$.
Thus, the $\sigma_2$-orbits of $\sigma_1^{-1}(w_0)$, $w_0$, and $\sigma_1(w_0)$ are distinct.
Since $|\sigma_2| < \infty$, there is exactly one nonmonochromatic $\sigma_2$-orbit under $c_2$ and hence Lemma~\ref{cl:gammaorder} applies, $|\sigma_2| = 3$, and $W' = \{\taur^{(1)}(w_0), \taub^{(1)}(w_0)\}$.
The orbits $\cO_0, \cO_1, \cO_2$ are distinct $\rotation1$-orbits, and hence $W'$ is disjoint from $\cO_i$ for some $i \in \{1,2\}$.
Thus $\cO_i$ is monochromatic under both $c_1$ and $c_2$, but 
\[
\blue = c_2(\sigma_1^{-1}(w_0)) = c_2(\sigma_2^i(\sigma_1^{-1}(w_0))) = c_2(\sigma_2^i(\sigma_1(w_0))) = c_2(\sigma_1(w_0)) = \red,
\]
a contradiction.
\end{proof}

We now have that $|\sigma_2| \leq |\sigma_1| < \infty$.
Observe that each $\rotation{i}$ is a rotation about the circle $x \mapsto x + \frac{j}{|\rotation{i}|}$ for some $j$ relatively prime to $|\rotation{i}|$. 
Therefore, $|\rotation1| = |\rotation2|$ if and only if $\rotation1$ generates $\rotation2$ and vice-versa.
We want to show that $\rotation{2}$ is generated by $\rotation{1}$.

\begin{lemma}\label{lma:BAM}
$|\sigma_1| = |\sigma_2|$ and hence $\sigma_1$ and $\sigma_2$ generate each other.
\end{lemma}

\begin{proof}
First note that by the extremal choice (C0), $|\sigma_2| = |\sigma_1|$ if and only if $\sigma_2$ generates $\sigma_1$.
Now suppose that $|\sigma_2| < |\sigma_1|$.
Let $S_j$ be the $\sigma_2$-orbit of $\sigma_1^j(w_0)$.
Since $|\cO_0| = |\sigma_1| >  |\sigma_2| = |S_0|$, the $\sigma_2$-orbits $S_0$ and $S_1$ are distinct; similarly $S_0$ and $S_{-1}$ are distinct.
Since $|\sigma_2|$ is finite, $S_0$ is the only nonmonochromatic $\sigma_2$-orbit under $c_2$.
Therefore, $S_1$ and $S_{-1}$ are monochromatic under $c_2$ and $S_{-1} \neq S_1$ since $c_2(\sigma_1^{-1}(w_0))\neq c_2(\sigma_1(w_0))$.
However, since $|\cO_0^{\red}|, |\cO_0^{\blue}| < \infty$, Lemmas~\ref{lma:notbothtaus} and~\ref{cl:gammaorder} imply that $S_{-1} = S_1$, a contradiction.
\end{proof}

We have now verified that $\rotation{1}$ and $\rotation{2}$ are both finite and generate each other.
Observe that this is a trivial statement in the case that $X$ is a cycle of prime order.
We finish the proof by showing that $\rotation{1}$ and $\rotation{2}$ violate either Fact~\ref{prop:sigmatraversal} or the extremal choices.

For all remaining cases, we will assume that $|\cO_0^{\red}|\leq |\cO_0^{\blue}|$.
The other case is symmetric by swapping colors and possibly negating the exponents on $\rotation{1}$.
The following lemma completes the proof of Theorem~\ref{thm:divisibility}.

\begin{lemma}\label{thm:finishit}
If $c_1$ and $c_2$ are selected by the extremal choices (C0)--(C2) and the definition of $W'$, then there exists a distinguishing coloring of $X$.
\end{lemma}

\begin{proof}
By Lemma~\ref{lma:BAM}, $\rotation{1}$ and $\rotation{2}$ generate each other.
Since $|\cO_0^{\red}|, |\cO_0^{\blue}| < \infty$, Lemma~\ref{lma:notbothtaus} implies that $W' = \{w'\}$.
This further implies that $w' \in \cO_0$, and since $w_0\in \cO_0 \setminus \{w'\}$, the Divisibility Condition implies that $|\cO_0| \geq 6$.
Let $\ell$ be the integer minimizing $|\ell|$ such that $\rotation{2} = \rotation{1}^\ell$.
Recall that $\sigma_1^{-1}(w_0),\sigma_1(w_0)\notin W'$ and thus $c_2(\sigma_1^{-1}(w_0)) = \blue$ and $c_2(\sigma_1(w_0)) = \red$.

Suppose first that $\ell\in\{1,\ldots,|\cO_0^{\red}|\}$.
Thus $c_1(\sigma_1^{\ell-1}(w_0))=\red$ and $c_1(\sigma_1^{1-\ell}(w_0)) = \blue$.
Since $\sigma_2(\{ \sigma_1^{1-\ell}(w_0), \sigma_1^{-1}(w_0)\}) = \{ \sigma_1(w_0), \sigma_1^{\ell-1}(w_0)\}$ and at most one of $\sigma_1^{\ell-1}(w_0)$ and $\sigma_1^{1-\ell}(w_0)$ is equal to $w'$, and thus $\sigma_2$ sends a blue element to a red element under $c_2$, contradicting Fact~\ref{prop:sigmatraversal}.

Now suppose that $-\ell\in\{1,\ldots,|\cO_0^{\red}|\}$.
Thus $c_1(\sigma_1^{\ell-1}(w_0))=\blue$ and $c_1(\sigma_1^{1-\ell}(w_0)) = \red$.
Since $\sigma_2(\{ \sigma_1^{-1}(w_0), \sigma_1^{1-\ell}(w_0)\}) = \{ \sigma_1^{\ell-1}(w_0), \sigma_1(w_0)\}$ and at most one of $\sigma_1^{\ell-1}(w_0)$ and $\sigma_1^{1-\ell}(w_0)$ is equal to $w'$, and thus $\sigma_2$ sends a blue element to a red element under $c_2$, contradicting Fact~\ref{prop:sigmatraversal}.

Therefore we may assume that $|\cO_0^{\red}| < |\ell| \leq \lfloor\frac{|\sigma_1|}{2}\rfloor$.
For all $x\in \cO_0^{\red}$, the element $\sigma_2^{-1}(x)$ is blue under $c_1$.
Thus $|\cO_0^{\red}| = 1$, but then there are two red elements in $\cO_0$ (which is the $\sigma_2$-orbit of $w_0$) under $c_2$, violating extremal choice (C2).
\end{proof}

Observe that the divisibility condition holds for any set $W \subset V(C_n)$ when the smallest prime divisor of $n$ is at least $7$, so Theorem~\ref{thm:main3} follows immediately from Theorem~\ref{thm:divisibility}.

\vspace{0.5em}
\noindent\textbf{Theorem~\ref{thm:main3}.} 
If the smallest prime divisor of $n$ is at least $7$, then $\extD(C_n) = 4$.
\vspace{0.5em}

In the other cases, we must start with a larger set of blanks.

\begin{lemma}\label{lma:16}
Let $W \subset \R/\Z$ be a set with $|W| \geq 16$.
There exists a set $W' \subset W$ where $|W'| = 4$ and $W'$ satisfies the Divisibility Condition.
\end{lemma}

\begin{proof}
Split the circle into five intervals given by $[ \frac{i}{5}, \frac{i+1}{5})$ for $i \in \{0,\dots,4\}$.
By the pigeonhole principle, there are at least four elements of $W$ in one of these intervals.
This set satisfies the Divisibility Condition.
\end{proof}

Note that Theorem~\ref{thm:divisibility} and Lemma~\ref{lma:16} imply Theorem~\ref{thm:main2}.

\vspace{0.5em}
\noindent\textbf{Theorem~\ref{thm:main2}.} $\extD(\R/\Z) \leq 16$.
\vspace{0.5em}

Observe that the set of 15 elements given by $\{ \frac{i}{15} : 0 \leq i \leq 14 \}$ contains no four elements that satisfy the Divisibility Condition, so Theorem~\ref{thm:main2} is the best upper bound that is implied by Theorem~\ref{thm:divisibility}.
Further, we get the following upper bound on $\extD(C_n)$ for general $n$.
Let $\chi_n(i)$ be the indicator function that equals 1 if and only if $i$ divides $n$.

\begin{corollary}
Let $n \geq 6$. Then $\extD(C_n) \leq 3 \left(1+\chi_n(2) + 2\chi_n(3) + 2\chi_n(4) + 4\chi_n(5)\right)+1$.
\end{corollary}

\begin{proof}
Let $W \subset V(C_n)$ have size $3(1+\chi_n(2) + 2\chi_n(3) + 2\chi_n(4) + 4\chi_n(5))+1$.
For each $w \in W$, there are at most $\chi_n(2) + 2\chi_n(3) + 2\chi_n(4) + 4\chi_n(5)$ elements in $W \cap \{ w + \frac{i}{k} : 1 \leq i < k, 2 \leq k \leq 5\}$.
Iteratively select a subset $W' = \{ w_1, w_2, w_3, w_4\} \subset W$ of size four where $w_\ell \neq w_j + \frac{i}{k}$ for all $\ell > j$ and $1 \leq i < k \leq 5$.
This restriction removes at most $1+ \chi_n(2) + 2\chi_n(3) + 2\chi_n(4) + 4\chi_n(5)$ elements from $W$ in each of the first three selections of $w_1$, $w_2$, and $w_3$.
At least one element remains to select $w_4$.
The set $W'$ satisfies the Divisibility Condition and hence any coloring on $V(C_n) \setminus W$ extends arbitrarily to $V(C_n) \setminus W'$ and can be distinguished using Theorem~\ref{thm:finishit}.
\end{proof}

\section{Future Work}
\label{sec:conclusion}

In addition to resolving Conjectures \ref{conj:main} and \ref{conj:S1}, we pose the following questions for future study.

\begin{problem}
	Determine $\extD(\R^k)$ for $k \geq 2$.
\end{problem}

\begin{problem}
	Determine $\extD(\S^d)$ for $d \geq 2$.
\end{problem}

Suppose $W \subset \S^d$ is a set of blanks and $c : \S^d\setminus W \to \{\red,\blue\}$ is a 2-coloring of $\S^d\setminus W$ that does not extend to a distinguishing coloring of $\S^d$. Observe that $W' = W \cup \{ 0 \} \subset \R^{d+1}$ is a set of blanks and 
$$c'(\vx) = \begin{cases}c(\vx) & \text{if $\vx \in \S^d$,}\\ \red &\text{otherwise}\end{cases}$$
is a 2-coloring of $\R^{d+1} \setminus W$ that does not extend to a distinguishing coloring of $\R^{d+1}$.  Therefore, $\extD(\S^d) < \extD(\R^d)$.  This, together with Lemma \ref{lemma:faithful2} and the observation that the 3-dimensional cube has a faithful embedding into $\S^2$, gives rise to the following conjectures, the first of which also follows from our conjecture that $\extD(\S^1)=6$.  

%
\begin{conjecture}
	$\extD(\R^2) = 7$.
\end{conjecture}

\begin{conjecture}
	$\extD(\S^2) = 9$.
\end{conjecture}

\begin{conjecture}
	$\extD(\R^3) = 10$.
\end{conjecture}

Finally, proper coloring and list-coloring versions of the distinguishing number were introduced in \cite{CT} and  \cite{FFG}, respectively.  We feel that studying these distinguishing parameters through the lens of precoloring extensions would be an interesting direction for further inquiry, in line with the broader precoloring extension literature discussed in the introduction.

\end{document}